\documentclass[12pt]{article}
\usepackage{amsmath,amsthm,amssymb,tikz}
\usetikzlibrary{decorations.pathreplacing}
\usepackage{mathtools}

\usepackage{color}
\usepackage{fullpage}
\usepackage{float}
 \usepackage{setspace}
\usepackage{epsf}
\usepackage[absolute]{textpos}
 \tikzset{
  mynode/.style={fill,circle,inner sep=2pt,outer sep=0pt}
}
\usepackage{graphicx}
\usepackage[colorlinks=true,linkcolor=webgreen,filecolor=webbrown,citecolor=webgreen]{hyperref}

\definecolor{webgreen}{rgb}{0,.5,0}
\definecolor{webbrown}{rgb}{.6,0,0}
\theoremstyle{plain}
\newtheorem{theorem}{Theorem}

\newtheorem{corollary}{Corollary}

\newtheorem{lemma}{Lemma}

\theoremstyle{definition}
\newtheorem{definition}{Definition}
\theoremstyle{remark}

\numberwithin{equation}{section}

\newcommand{\h}{\Bar{h}}

\newcommand{\field}{\mathbb{F}}
\newcommand{\F}{\mathcal{F}}

\newcommand{\RA}{\mathcal{R}}
\newcommand{\DR}{\mathcal{DR}}

\usepackage{xcolor}

\title{Finding an Isomorphism between the Riordan Group and a Subgroup of the Double Riordan Group}
\author{Shakuan K. Frankson\\
\small Mathematics Department\\[-0.8ex]
\small Howard University, Washington, DC\\
\small \texttt{shakuan.frankson@bison.howard.edu}\\
\\
\hspace{0cm} \small Mathematics Subject Classifications: 05A05, 05A10}
\date{}
\begin{document}
\maketitle
\begin{abstract}
The Riordan group is a set of infinite lower-triangular matrices defined by two generating functions, $g$ and $f$. The elements of the group are called Riordan arrays, denoted by $(g,f)$, and the $k$th column of a Riordan array is given by the function $gf^k$. The Double Riordan group is defined similarly using three generating functions $g$, $f_1$, and $f_2$, where $g$ is an even function and $f_1$ and $f_2$ are odd functions. This group generalizes the Checkerboard subgroup of the Riordan group, where $g$ is even and $f$ is odd. An open question posed by Davenport, Shapiro, and Woodson \cite{dr} was if there exists an isomorphism between the Riordan group and a subgroup of the Double Riordan group. This question is answered in this article.
\end{abstract}


\pagenumbering{arabic}

\section{Introduction}
Before we discuss the algebraic structure of the Riordan group, we must set out the terminology needed to define its elements. 
\begin{definition}
Given a sequence $(a_n)=(a_0,a_1,a_2,a_3,...)$, the formal power series \begin{equation*}
A(z)=a_0+a_1z+a_2z^2+a_3z^3+...=\sum_{k\geq0} a_kz^k 
\end{equation*}
is called the \textbf{ordinary generating function} of a sequence $(a_n)$. The notation $[z^k]A(z)$ refers to the $k$th coefficient of the formal power series, $a_k$. Formally, $[z^k]$ is called the \textbf{coefficient extraction operator} of a formal power series.
\end{definition}

From these definitions, we now look at the set of matrices that Shapiro, Getu, Woan, and Woodson define in \cite{original}.
\begin{definition}
Consider infinite matrices $M=(m_{i,j})$ over a field $\field$, for ${i,j\geq0}$. Let $$C_k(z)=\sum_{n=0}^\infty m_{n,k}z^n$$ be the column generating function of the $k$th column of $M$. When each column generating function of $M$ is given by
\begin{equation*}
C_k(z)=g(z)f(z)^k
\end{equation*}
for some formal power series $g(z)$ and $f(z)$, where
\begin{equation*}
g(z)=g_0+g_1z+g_2z^2+g_3z^3+...=\sum_{k=0}^\infty g_kz^k,\ \ g_0\neq0
\end{equation*}
and 
\begin{equation*}
f(z)=f_1z+f_2z^2+f_3z^3+...=\sum_{k=1}^\infty f_kz^k,\ \ f_1\neq0, 
\end{equation*}
the matrix $M$ is called a \textbf{Riordan matrix} (or, equivalently, a \textbf{Riordan array}). The matrix is denoted $M=(g(z),f(z))$ or $(g,f)$, if it does not cause confusion.
\end{definition}

\begin{definition}
The \textbf{(ordinary) Riordan group} $\mathcal{R}$ is a set of infinite lower-triangular matrices that are defined by two (ordinary) generating functions, $g(z)$ and $f(z)$, where $g_0\neq0$ and $f_1\neq0$. The function $f(z)$ is referred to as the \textbf{multiplier function}, and the elements are called Riordan arrays or Riordan matrices.
\end{definition}


To determine the product of two Riordan arrays, we first observe the product of a Riordan array with an infinite column vector associated with some ordinary generating function $A(z)$. The result is another infinite column vector given by the generating function $B(z)$. This important observation is the Fundamental Theorem of Riordan Arrays (FTRA) \cite{original}.
\begin{theorem}
\textbf{(The Fundamental Theorem of Riordan Arrays):} Let $%
A(z)=\sum_{k=0}^{\infty }{a_{k}z^{k}}$ and $B(z)=\sum_{k=0}^{\infty }{%
b_{k}z^{k}}$ and let $A$ and $B$ be the column vectors $A=\left(
a_{0},a_{1},a_{2},\cdots \right) ^{T}$ and $B=\left(
b_{0},b_{1},b_{2},\cdots \right) ^{T}$. Then $(g,f)*A=B$, if and only if $B(z)=g(z)A(f(z))$.
\end{theorem}
\begin{proof}
\vspace{-0.1cm}
\begin{eqnarray*}
\left[ g\ \ gf\ \ gf^{2}\ \ gf^3\ \ \cdots \right] 
\begin{bmatrix}
a_{0} \\ 
a_{1} \\ 
a_{2} \\ 
\vdots%
\end{bmatrix}
&=&%
\begin{bmatrix}
b_{0} \\ 
b_{1} \\ 
b_{2} \\ 
\vdots%
\end{bmatrix}
\\
a_{0}g+a_{1}gf+a_{2}gf^{2}+\cdots &=&B\left( z\right) \\
g\left( a_{0}+a_{1}f+a_{2}f^{2}+a_{3}f^{3}+\cdots \right) &=& B(z)\\
g(z)\ast A\left( f(z)\right) &=&B\left( z\right)
\end{eqnarray*}\vspace{0cm}
\end{proof}

Using the FTRA, we extend to the multiplication rule for Riordan arrays.
\begin{theorem}
Let $(g,f)$ and $(G,F)$ be Riordan arrays. Then, the product of the matrices is given by
\begin{equation*}
\left( g,f\right)\ast \left( G,F\right) =\left( gG\left( f\right) ,F\left(
f\right) \right) .
\end{equation*}
\end{theorem}

Associativity holds as a result of the matrix multiplication. The inverse of a Riordan array $(g,f)$ is given by
\begin{equation*}
\left( g,f \right) ^{-1}=\left( \frac{1}{g\left( \overline{f}\right) },\overline{f}\right) .
\end{equation*}
To satisfy the group axioms, we establish that the identity of the Riordan group is the infinite identity matrix, which we denote by $\left( 1,z\right).$  \\

Now, consider the case when an array has two multiplier functions. To preserve the group structure, we set further conditions for the generating functions to satisfy.
\begin{definition}\cite{dr}
Let $g(z)\in\F_0$ be an even function and $f_1(z),f_2(z)\in\F_1$ be odd functions. Define the initial column by the generating function $g(z)$ and multiply $f_1(z)$ and $f_2(z)$ alternately to each successive column. Then, the resulting array is called a \textbf{Double Riordan array}, and these arrays are elements of the \textbf{Double Riordan group}. We denote the group by $\mathcal{DR}$ or $2\RA$. Double Riordan matrices are denoted by $(g,f_1,f_2)$. The columns of the array are given by 
\begin{equation*}
\begin{array}{r}
(g,f_1,f_2)=\Big[g\ \ gf_{1}\ \ g(f_{1}f_{2})\ \ gf_{1}(f_1f_{2})\ \ g(f_{1}f_{2})^{2}\ \ \cdots \Big].\ 
\end{array}%
\end{equation*}%
\end{definition}

When the multiplier functions $f_1=f_2$, the Double Riordan array is an aerated Riordan array in the Checkerboard subgroup.

To find the product of two Double Riordan arrays, Shapiro, Woodson, and Davenport establish the Fundamental Theorem of Double Riordan Arrays (FTDRA) in \cite{dr}. When multiplying a $\mathcal{DR}$ array by an infinite column vector $A(z)$, we must consider two cases: when $A(z)$ is an even or odd function.

\begin{theorem}
\textbf{(The Fundamental Theorem of Double Riordan Arrays)} Let $%
g(z)=\sum_{k=0}^{\infty }{g_{2k}z^{2k}}$, $f_{1}(z)=\sum_{k=0}^{\infty }{%
f_{1,2k+1}z^{2k+1}}$, and $f_{2}(z)=\sum_{k=0}^{\infty }{f_{2,2k+1}z^{2k+1}}$%
.

Case 1: If $A(z)=\sum_{k=0}^{\infty }{a_{2k}z^{2k}}$ and $
B(z)=\sum_{k=0}^{\infty }{b_{2k}z^{2k}}$, and  $A=\left( a_{0},0,a_{2},0,\cdots \right) ^{T}$ and $B=\left(
b_{0},0,b_{2},0,\cdots \right) ^{T}$ are column vectors. Then,
\begin{center}
    $(g,f_{1},f_{2})*A$ $=B$ if and only if $B(z)=g(z)A(\sqrt{f_1f_2}).$
\end{center}

Case 2: If $A(z)=\sum_{k\geq0}{a_{2k+1}z^{2k+1}}$ and $%
B(z)=\sum_{k\geq0}{b_{2k+1}z^{2k+1}}$, then
\begin{center}
    $(g,f_{1},f_{2})*A=B$ if and only if $B(z) = g(z)\sqrt{\frac{f_1}{f_2}}A(\sqrt{f_1f_2})$.
\end{center}
\end{theorem}
\noindent Shapiro et al. generalize the FTDRA to the product of two Double Riordan Arrays and get the following theorem.
\begin{theorem} Let $(g,f_1,f_2)$ and $(G,F_1,F_2)$ be Double Riordan arrays. Then, the product of the arrays is given by
\begin{equation*}
    (g,f_{1},f_{2})*(G,F_{1},F_{2})=\left(gG(h),\frac{f_1}{h}F_{1}(h),\frac{f_2}{h}F_{2}(h)\right),
\end{equation*}
where $h=\sqrt{f_{1}f_{2}}$.
\end{theorem}

The generating function $h$ is an element of $\F_1$, so there exists a unique compositional inverse that is denoted by $\bar{h}$. Then, for some Double Riordan array $(g,f_1,f_2)$, its inverse is given by
\begin{equation*}
(g,f_1,f_2)^{-1} = \left(\frac{1}{g(\Bar{h})},\ \frac{z\Bar{h}}{f_1(\Bar{h})},\ \frac{z\Bar{h}}{f_2(\Bar{h})}\right).
\end{equation*}
The identity element of the Double Riordan group is also the infinite identity matrix, which is denoted by $(1,z,z)$.

\section{The Isomorphism Problem}
Davenport, Shapiro, and Woodson \cite{dr} state that there exists an isomorphic mapping $\psi:\RA_C\rightarrow\DR$ defined by
\begin{equation*}
    \psi:(g,f)\mapsto(g,f,f),
\end{equation*}
where $\RA_C$ is the Checkerboard subgroup of the Riordan group, $g(z)$ is an even function, and $f(z)$ is an odd function. An open question from \cite{dr} asks if there exists an isomorphism between the Riordan group and a subgroup of the Double Riordan group. We answer this question in the following section.

\subsection{The Monomorphism Theorems} \label{subsection 2.1}
Before we define these monomorphisms, we first establish the following lemma.
\begin{lemma}
The subset 
\begin{equation*}
S=\biggl\{\left(g(z^2),z,\frac{f(z^2)}{z}\right): g(z)\in\F_0\ \text{and}\ f(z)\in\F_1 \biggl\}
\end{equation*}
of the Double Riordan group is a subgroup.
\end{lemma}

\begin{proof}
Suppose $\left(g(z^2),z,\frac{f(z^2)}{z}\right)$ and $\left(G(z^2),z,\frac{F(z^2)}{z}\right)$ are elements of the set $S$. Then, the product
\begin{equation*}
\left(g(z^2),z,\frac{f(z^2)}{z}\right)*\left(G(z^2),z,\frac{F(z^2)}{z}\right)=\left(g(z^2)G\left(f(z^2)\right),z,\frac{F\left(f(z^2)\right)}{z}\right).
\end{equation*}
Thus, the closure property is satisfied. It now suffices to show the inverse of matrices in $S$ also lie in the set. Taking the inverse of some array $\left(g(z^2),z,\frac{f(z^2)}{z}\right)$, we see that
\begin{equation*}
\left(g(z^2),z,\frac{f(z^2)}{z}\right)^{-1}=\left(\frac{1}{g(\h^2)},\frac{z\h}{\h},\frac{z\h}{\frac{f(\h^2)}{\h}}\right)=\left(\frac{1}{g(\h^2)},z,\frac{z\h^2}{f(\h^2)}\right),
\end{equation*}
where $\h=\sqrt{f(z^2)}$. We note that
\begin{equation*}
\h^2=\left(\overline{\sqrt{f(z^2)}}\right)^2=\left({\sqrt{\bar{f}(z^2)}}\right)^2=\Bar{f}(z^2).
\end{equation*}
Thus, 
\begin{equation*}
\left(\frac{1}{g(\h^2)},z,\frac{z\h^2}{f(\h^2)}\right)=\left(\frac{1}{g(\h^2)},z,\frac{z\Bar{f}(z^2)}{f(\Bar{f}(z^2))}\right)=\left(\frac{1}{g(\h^2)},z,\frac{\Bar{f}(z^2)}{z}\right).
\end{equation*}
\end{proof}
\begin{definition}
The subset $S$ of Double Riordan arrays defined in Lemma $2.1$ is called the \textbf{type-$1$ almost Appell subgroup}. This subgroup is also represented by the set 
\begin{equation*}
\{(g,z,f_2): g(z)\ \text{is an even function},\ f_2(z)\ \text{is an odd function} \}. 
\end{equation*}
The type-$1$ almost Appell subgroup is not normal, but it does contain the normal Appell subgroup $\{(g,z,z):g(z)\ \text{is an even function}\}$. Furthermore, the type-$1$ almost Appell subgroup is isomorphic to the Riordan group for a certain map $\phi$, which we define in the following theorem.
\end{definition}
\begin{theorem}(Monomorphism Theorem I)
The type-$1$ almost Appell subgroup of the Double Riordan group is an isomorphic copy of the Riordan group, and the map $\phi:\RA\rightarrow\DR$ defined by
\begin{equation*}
\phi(g,f)=\left(g(z^2),z,\frac{f(z^2)}{z}\right)
\end{equation*}
is a group monomorphism.
\end{theorem}

\begin{proof}
Let $(g,f)$ and $(G,F)$ be two Riordan arrays. To show that the map $\phi$ is a homomorphism, we prove that
\begin{equation}\label{homomorph}
\phi(g,f)\ast \phi(G,F)=\phi\left((g,f)*(G,F)\right).
\end{equation}
Note that, for any Riordan array $(g,f)$, the generating functions $g(z)\in\F_0$ and $f(z)\in\F_1$. Thus, the functions $g(z^2)$ and $f(z^2)$ are even. Then, the lefthand side of Equation \eqref{homomorph} is given by
\begin{align*}
\phi(g,f)\ast \phi(G,F)=\left(g(z^2),z,\frac{f(z^2)}{z}\right)*\left(G(z^2),z,\frac{F(z^2)}{z}\right)\\
=\left(g(z^2)G\left(\left(\sqrt{f(z^2)}\right)^2\right),\sqrt{\frac{z^2}{f(z^2)}}\sqrt{f(z^2)},\sqrt{\frac{f(z^2)}{z^2}}\frac{F\left(\left(\sqrt{f(z^2)}\right)^2\right)}{\sqrt{f(z^2)}}\right)\\
=\left(g(z^2)G\left(f(z^2)\right),z,\frac{F\left(f(z^2)\right)}{z}\right)
\end{align*}
The righthand side of Equation \eqref{homomorph} is given by
\begin{align*}
\phi\left((g,f)*(G,F)\right)=\phi(gG(f),F(f))=\left(g(z^2)G(f(z^2)),z,\frac{F(f(z^2))}{z}\right).
\end{align*}
Thus, $\phi$ is a homomorphism. To prove the claim that $\phi$ is a monomorphism, we show that $\phi$ is also one-to-one. Assume that, for some Riordan arrays $(g,f)$ and $(G,F)$, the mapping $\phi(g,f)=\phi(G,F)$. Then,
\begin{equation*}
\left(g(z^2),z,\frac{f(z^2)}{z}\right)=\left(G(z^2),z,\frac{F(z^2)}{z}\right).
\end{equation*}
Clearly, this implies that the functions $g(z^2)=G(z^2)$ and $f(z^2)=F(z^2)$. Therefore, $(g,f)=(G,F)$. Thus, the claim is true.
\end{proof}

It is not difficult to show that the map $\psi:\RA\rightarrow\DR$ defined by $$\psi(g,f)=\left(g(z^2),\frac{f(z^2)}{z},z\right)$$ is also a group monomorphism. We define the set 
\begin{align*}
\biggl\{\left(g(z^2),\frac{f(z^2)}{z},z\right):g(z)\in\F_0, f(z)\in\F_1\biggl\}\\
=\biggl\{\left(g,f_1,z\right):\ \text{$g$ is an even function, $f_1$ is an odd function}\biggl\}
\end{align*}
to be the \textbf{type-2 almost Appell subgroup} of the Double Riordan group. Therefore, we know that at least two copies of the Riordan group are found in the Double Riordan group. Extending this mapping to the $k$-Riordan group, we get the following corollary.
\begin{corollary}
Denoting the $k$-Riordan group by $k\RA$, the map $\phi_k:\RA\rightarrow k\RA$ defined by
\begin{equation*}
\phi_k(g,f)=\left(g(z^k),\underbrace{z,z,...,z}_{\text{$k-1$}},\frac{f(z^k)}{z^{k-1}}\right)
\end{equation*}
is a group monomorphism. Furthermore, the map $\phi_k:\RA\rightarrow k\RA$ is a group monomorphism for any rearrangement of the $k$ multiplier functions. The Riordan group has $k$ known isomorphic copies in the $k$-Riordan group.
\end{corollary}

\begin{proof}
Let $(g,f)$ and $(G,F)$ be Riordan arrays, and let $h=\sqrt[k]{f(z^k)}$. Note that this implies $h^k=f(z^k)$. Then,
\begin{align*}
\phi_k(g,f)*\phi_k(G,F)=\left(g(z^k),\underbrace{z,z,...,z}_{\text{$k-1$}},\frac{f(z^k)}{z^{k-1}}\right)*\left(G(z^k),\underbrace{z,z,...,z}_{\text{$k-1$}},\frac{F(z^k)}{z^{k-1}}\right)\\
=\left(g(z^k)G(f(z^k)),\underbrace{\frac{z}{h}h,...,\frac{z}{h}h}_{\text{$k-1$}},\frac{\frac{f(z^k)}{z^{k-1}}}{h}\frac{F(f(z^k))}{h^{k-1}}\right)=\left(g(z^k)G(f(z^k)),\underbrace{z,...,z}_{\text{$k-1$}},\frac{f(z^k)}{z^{k-1}}\frac{F(f(z^k))}{h^{k}}\right)\\
= \left(g(z^k)G(f(z^k)),\underbrace{z,...,z}_{\text{$k-1$}},\frac{F(f(z^k))}{z^{k-1}}\right).
\end{align*}
Also,
\begin{align*}
\phi_k\left((g,f)*(G,F)\right)=\phi_k(gG(f),F(f))=\left(g(z^k)G(f(z^k)),\underbrace{z,z,...,z}_{\text{$k-1$}},\frac{F(f(z^k))}{z^{k-1}}\right).
\end{align*}
Thus, $\phi_k$ is group homomorphism. It is straightforward to see the homomorphism is also one-to-one. The image of the map is the set
$$S_k=\biggl\{\Big(g(z^k),\underbrace{z,z,...,z}_{\text{$k-1$}},\frac{f(z^k)}{z^{k-1}}\Big):g(z)\in\F_0,\ \text{and}\ f(z)\in\F_1\biggl\}.$$ Rearranging the multiplier functions of elements in $S_k$ and mapping the Riordan group to that set is also a group monomorphism. 
\end{proof}

For $k\geq3$, are there isomorphic copies of the $(k-1)$-Riordan group in the $k$-Riordan group? If so, which copies can we define? Will the isomorphism hold if we permute the multiplier functions of the arrays in the codomain? We answer this question in the following theorems.
\begin{lemma}
Denote the Triple Riordan ($3$-Riordan) Group by $3\RA$, and let the generating functions $g(z)\in\F_0$ and $f_1(z),f_2(z)\in\F_1$. Then, the subset
\begin{equation*}
\biggl\{\left(g(z^3),z,\frac{f_1(z^3)}{z^2},\frac{f_2(z^3)}{z^2}: g(z)\in\F_0\ \text{and}\ f_1(z),f_2(z)\in\F_1 \right)\biggl\}
\end{equation*}
is a subgroup of the $3$-Riordan group.
\end{lemma}

\begin{proof}
The proof is similar to the proof of Lemma $2.1$.
\end{proof}

\begin{definition}
The subgroup defined in Lemma $2.2$ is called the \textbf{type-1 almost Appell subgroup} of the Triple Riordan group. Permuting $z$ such that it is the second or third multiplier function listed generates the \textbf{type-2 almost Appell subgroup} and the \textbf{type-3 almost Appell subgroup}, respectively.
\end{definition}

\begin{theorem}(Monomorphism Theorem II)
Let $g(z)\in\F_0$ and $f_1(z),f_2(z)\in\F_1$. Then, the matrix
\begin{equation}\label{dr to 3r}
    \left(g(z^2),\frac{f_1(z^2)}{z},\frac{f_2(z^2)}{z} \right)
\end{equation}
is a Double Riordan array. The map $\chi:\DR\rightarrow3\RA$ defined by 
\begin{equation*}
\chi\left(g(z^2),\frac{f_1(z^2)}{z},\frac{f_2(z^2)}{z} \right)=\left(g(z^3),z,\frac{f_1(z^3)}{z^2},\frac{f_2(z^3)}{z^2}\right)
\end{equation*}
is a group monomorphism. Furthermore, the Double Riordan group is isomorphic to the type-$1$ almost Appell subgroup of the Triple Riordan group. The Double Riordan group is also isomorphic to the type-$2$ and type-$3$ almost Appell subgroups. 
\end{theorem}

\begin{proof}
Select generating functions $g(z),G(z)\in\F_0$ and $f_1(z),f_2(z),F_1(z),F_2(z)\in\F_1$. Then, the matrices $D_1=\left(g(z^2),\frac{f_1(z^2)}{z},\frac{f_2(z^2)}{z} \right)$ and $D_2=\left(G(z^2),\frac{F_1(z^2)}{z},\frac{F_2(z^2)}{z} \right)$ are Double Riordan arrays. 
Also, let $h(z^3)=\sqrt[3]{\frac{f_1(z^3)f_2(z^3)}{z^3}}$ and note that $h^3(z^3)=\frac{{f_1(z^3)f_2(z^3)}}{z^3}$. To show that $\alpha$ is a group homomorphism, we must show that $\chi(D_1)\cdot\chi(D_2)=\chi(D_1*D_2)$. 

The lefthand side of Equation \eqref{dr to 3r} is given by
\begin{align*}
\chi(D_1)\cdot\chi(D_2)=\chi\left(g(z^2),\frac{f_1(z^2)}{z},\frac{f_2(z^2)}{z} \right)\cdot\chi\left(G(z^2),\frac{F_1(z^2)}{z},\frac{F_2(z^2)}{z} \right)\\
=\left(g(z^3)G(h^3(z^3)), \frac{zh(z^3)}{h(z^3)}, \frac{\frac{f_1(z^3)}{z^2}}{h(z^3)} \frac{F_1(h^3(z^3))}{h^2(z^3)}, \frac{\frac{f_2(z^3)}{z^2}}{h(z^3)} \frac{F_2(h^3(z^3))}{h^2(z^3)} \right)\\
=\left(g(z^3)G(h^3(z^3)), z, \frac{\frac{f_1(z^3)}{z^2}}{h^3(z^3)} F_1(h^3(z^3)), \frac{\frac{f_2(z^3)}{z^2}}{h^3(z^3)} F_2(h^3(z^3))\right)\\
=\left(g(z^3)G(h^3(z^3)), z, \frac{f_1(z^3)}{z^2}\frac{z^3}{f_1(z^3)f_2(z^3)} F_1(h^3(z^3)), \frac{f_2(z^3)}{z^2}\frac{z^3}{f_2(z^3)f_2(z^3)} F_2(h^3(z^3))\right) \\
=\left(g(z^3)G(h^3(z^3)), z, \frac{zF_1(h^3(z^3))}{f_2(z^3)} ,\ \frac{zF_2(h^3(z^3))}{f_1(z^3)} \right) 
\end{align*}

We compare this to the righthand side of Equation \eqref{dr to 3r}. Let $p(z^2)=\sqrt{\frac{f_1(z^2)f_2(z^2)}{z^2}}$. Then,
\begin{align*}
\chi(D_1*D_2)=\chi\left(\left(g(z^2),z,\frac{f_1(z^2)}{z},\frac{f_2(z^2)}{z} \right)*\left(G(z^2),z,\frac{F_1(z^2)}{z},\frac{F_2(z^2)}{z} \right)\right)\\
= \chi\left( g(z^2)G(p^2(z^2)),\ \frac{zF_1(p^2(z^2))}{f_2(z^2)},\ \frac{zF_1(p^2(z^2))}{f_1(z^2)} \right)\\
=\chi\left( g(z^2)G(p^2(z^2)),\ \frac{1}{z}\frac{z^2F_1(p^2(z^2))}{f_2(z^2)},\ \frac{1}{z}\frac{z^2F_1(p^2(z^2))}{f_1(z^2)} \right)\\
=\left(g(z^3)G(h^3(z^3)),z,\frac{1}{z^2}\frac{z^3F_1(h^3(z^3))}{f_2(z^3)},\frac{1}{z^2}\frac{z^3F_2(h^3(z^3))}{f_1(z^3)} \right)\\
=\left(g(z^3)G(h^3(z^3)),z,\frac{zF_1(h^3(z^3))}{f_2(z^3)},\frac{zF_2(h^3(z^3))}{f_1(z^3)} \right)
\end{align*}

Thus, $\chi$ is a group homomorphism. It is straightforward to verify that the map is one-to-one. So, the Double Riordan group is isomorphic to the type-$1$ almost Appell subgroup of the Triple Riordan group. A similar argument can be used to show that the Double Riordan group is also isomorphic to the type-$2$ and type-$3$ almost Appell subgroups.
\end{proof}

So far, we have shown that at least two isomorphic copies of the Riordan group are in the Double Riordan group, and at least three isomorphic copies of the Double Riordan group are in the Triple Riordan group. Generalizing these results, we get the theorem below.

\begin{theorem}(Generalized Monomorphism Theorem)
Let $k$ be a nonnegative integer. Then, there are at least $k+1$ copies of the $k$-Riordan group in the $(k+1)$-Riordan group.
\end{theorem}

\begin{proof}
Let $g(z)\in\F_0$ and $f_1(z),...,f_k(z)\in\F_1$. Define the maps $\chi_i:k\RA\rightarrow(k+1)\RA$ for integers $1\leq i\leq k+1$ such that
\begin{equation*}
\begin{cases}
\chi_1\left(g(z^k),\frac{f_1(z^k)}{z^{k-1}},...,\frac{f_k(z^k)}{z^{k-1}}\right)=\Big(g(z^{k+1}),z,\frac{f_1(z^{k+1})}{z^k},...,\frac{f_k(z^{k+1})}{z^k}\Big)\\
\chi_2\left(g(z^k),\frac{f_1(z^k)}{z^{k-1}},...,\frac{f_k(z^k)}{z^{k-1}}\right)=\Big(g(z^{k+1}),\frac{f_1(z^{k+1})}{z^k},z,...,\frac{f_k(z^{k+1})}{z^k}\Big)\\
\indent {\vdots}\\
\chi_{k+1}\left(g(z^k),\frac{f_1(z^k)}{z^{k-1}},...,\frac{f_k(z^k)}{z^{k-1}}\right)=\Big(g(z^{k+1}),\frac{f_1(z^{k+1})}{z^k},...,\frac{f_k(z^{k+1})}{z^k},z\Big),
\end{cases}
\end{equation*}
where $\chi_i$ fixes $z$ in the $(i+1)^{th}$ position in the list of generating functions of the the $(k+1)$-Riordan array. Theorems $2.1$ and $2.2$ are specific cases of the proposition when $k=1$ and $k=2$, respectively. We show the proof that $\chi_1$ is a group homomorphism and the remaining maps will follow a similar argument. Let $\left(g(z^k),\frac{f_1(z^k)}{z^{k-1}},...,\frac{f_k(z^k)}{z^{k-1}}\right)$ and $\left(G(z^k),\frac{F_1(z^k)}{z^{k-1}},...,\frac{F_k(z^k)}{z^{k-1}}\right)$ be $(k+1)$-Riordan arrays. Then,
\begin{align*}
\chi_1\left(g(z^k),\frac{f_1(z^k)}{z^{k-1}},...,\frac{f_k(z^k)}{z^{k-1}}\right)*\chi_1\left(G(z^k),\frac{F_1(z^k)}{z^{k-1}},...,\frac{F_k(z^k)}{z^{k-1}}\right)\\
=\Big(g(z^{k+1}),z,\frac{f_1(z^{k+1})}{z^k},...,\frac{f_k(z^{k+1})}{z^k}\Big)*\Big(G(z^{k+1}),z,\frac{F_1(z^{k+1})}{z^k},...,\frac{F_k(z^{k+1})}{z^k}\Big)
\end{align*}
Let 
\begin{equation*}
h(z^{k+1})=\sqrt[k+1]{\frac{\prod_{j=1}^{k}f_j(z^{k+1})}{z^{(k+1)(k-1)}}},\ \ \text{which implies that}\ \ h^{k+1}(z^{k+1})=\frac{\prod_{j=1}^{k}f_j(z^{k+1})}{z^{(k+1)(k-1)}}
\end{equation*}
for the $(k+1)$-Riordan array $\Big(g(z^{k+1}),z,\frac{f_1(z^{k+1})}{z^k},...,\frac{f_k(z^{k+1})}{z^k}\Big)$. Then,
\begin{align*}
\Big(g(z^{k+1}),z,\frac{f_1(z^{k+1})}{z^k},...,\frac{f_k(z^{k+1})}{z^k}\Big)*\Big(G(z^{k+1}),z,\frac{F_1(z^{k+1})}{z^k},...,\frac{F_k(z^{k+1})}{z^k}\Big)\\
=\left(g(z^{k+1})G\left(h^{k+1}(z^{k+1})\right),\frac{z}{h(z^{k+1})}h(z^{k+1}),\frac{\frac{f_1(z^{k+1})}{z^k}}{h(z^{k+1})}\frac{F_1(h^{k+1}(z^{k+1}))}{h^k(z^{k+1})},...,\frac{\frac{f_k(z^{k+1}(z^{k+1}))}{z^k}}{h(z^{k+1})}\frac{F_k(h^{k+1}(z^{k+1}))}{h^k(z^{k+1})}\right)\\
=\left(g(z^{k+1})G\left(h^{k+1}(z^{k+1})\right),z,\frac{f_1(z^{k+1})}{z^k}\frac{F_1(h^{k+1}(z^{k+1}))}{h^{k+1}(z^{k+1})},...,\frac{f_k(z^{k+1})}{z^k}\frac{F_k(h^{k+1}(z^{k+1}))}{h^{k+1}(z^{k+1})}\right).
\end{align*}

In comparison, let 
\begin{equation*}
j(z^k)=\sqrt[k]{\frac{\prod_{i=1}^k f_i(z^k)}{z^{k(k-1)}}},\ \ \text{which implies that}\ \ j^k(z^{k+1})=\frac{\prod_{i=1}^k f_i(z^{k+1})}{z^{(k+1)(k-1)}}=h^{k+1}(z^{k+1}).
\end{equation*}
Then,
\begin{align*}
\chi_1\left(\left(g(z^k),\frac{f_1(z^k)}{z^{k-1}},...,\frac{f_k(z^k)}{z^{k-1}}\right)*\left(G(z^k),\frac{F_1(z^k)}{z^{k-1}},...,\frac{F_k(z^k)}{z^{k-1}}\right)\right)\\
=\chi_1\left(g(z^k)G(j^k(z^k)),\frac{\frac{f_1(z^k)}{z^{k-1}}}{j(z^k)}\frac{F_1(j^k(z^k))}{j^{k-1}(z^k)},...,\frac{\frac{f_k(z^k)}{z^{k-1}}}{j(z^k)}\frac{F_k(j^k(z^k))}{j^{k-1}(z^k)}\right)\\
=\chi_1\left(g(z^k)G(j^k(z^k)),\frac{f_1(z^k)}{z^{k-1}}\frac{F_1(j^k(z^k))}{j^{k}(z^k)},...,\frac{f_k(z^k)}{z^{k-1}}\frac{F_k(j^k(z^k))}{j^{k}(z^k)}\right)\\
=\left(g(z^{k+1})G(j^k(z^{k+1})),z,\frac{f_1(z^{k+1})}{z^{k}}\frac{F_1(j^k(z^{k+1}))}{j^{k}(z^{k+1})},...,\frac{f_k(z^{k+1})}{z^{k}}\frac{F_k(j^{k}(z^{k+1}))}{j^{k}(z^{k+1})}\right)\\
=\left(g(z^{k+1})G\left(h^{k+1}(z^{k+1})\right),z,\frac{f_1(z^{k+1})}{z^k}\frac{F_1(h^{k+1}(z^{k+1}))}{h^{k+1}(z^{k+1})},...,\frac{f_k(z^{k+1})}{z^k}\frac{F_k(h^{k+1}(z^{k+1}))}{h^{k+1}(z^{k+1})}\right).
\end{align*}

\end{proof}

\section{Conclusion}
An open question that remains is what other isomorphisms can be defined between the Riordan group and a subgroup of the Double Riordan group. Also, how can we use these maps to find more isomorphisms between the Riordan group and the $k$-Riordan group?

\end{document}